\documentclass{amsart}
\usepackage{amssymb}
\newtheorem{Prop}{Proposition}
\newtheorem{Theo}{Theorem}
\newtheorem{Lem}{Lemma}

\newcommand{\Hom}{\mathrm{Hom}}

\newcommand{\F}{\mathbb{F}}
\begin{document}
\title{The subgroup growth spectrum of virtually free groups}
\author[J.-C. Schlage-Puchta]{Jan-Christoph Schlage-Puchta}
\begin{abstract}
For a finitely generated group $\Gamma$ denote by $\mu(\Gamma)$ the growth
coefficient of $\Gamma$, that is, the infimum over all real numbers $d$
such that $s_n(\Gamma)<n!^d$. We show that the growth coefficient of a
virtually free group is always rational, and that every rational
number occurs as growth coefficient of some virtually free group.
\end{abstract}
\maketitle
For a finitely generated group $\Gamma$ denote by $s_n(\Gamma)$ the number of
subgroups of index $n$ in $\Gamma$, and define the growth coefficient
$\mu(\Gamma)$ as
\[
\mu(\Gamma)=\limsup_{n\to\infty} \frac{\log s_n(\Gamma)}{n\log n}.
\]
This quantity has been computed in a
variety of cases, including free product of finite groups
\cite{MInvent}, Fuchsian groups \cite{Tri}, and certain one-relator
groups \cite{onerel}. In all these examples, $\mu(\Gamma)$ can be expressed as
$-1+\sum_{i=1}^k (1-\frac{1}{n_i})$ for certain integers $n_i\geq 2$. This
observation led to the question as to which values $\mu(\Gamma)$ can attain; in
particular, in \cite{Examples} it was asked whether there exists a
sequence of groups $(\Gamma_n)$, such that $\mu(\Gamma_n)$ converges from
above. This question is answered by the following Theorem.
\begin{Theo}
\label{thm:Main}
The set $\{\mu(\Gamma): \Gamma\mbox{ finitely generated, virtually free}\}$ is equal
to the set of non-negative rational numbers.
\end{Theo}
The proof of this result falls into two steps, the first one being the
construction of virtually free groups with given growth coefficient,
these groups are described in the following theorem.
\begin{Theo}
\label{thm:Amalgam}
Let $k, \ell$ be integers, $p$ a prime number, and suppose that $k-\ell p\geq
2$. Define the group 
$\Gamma_{p, k, \ell}$ as the amalgamated product $S_k\ast_{C_p}S_k$, where a
generator of $C_p$ is mapped onto an element consisting of $\ell$ cycles
of length $p$ in both groups. Similarly, for $p$ odd or $p=2$ and $\ell$
even, let $\Gamma^+_{p, k, \ell}$ be the amalgamated product $A_k\ast_{C_p} A_k$. Then we have
\[
\mu(\Gamma_{p, k, \ell}) = 1-\frac{(p-1)\ell+1+\delta}{k},
\]
where $\delta=1$ if $p=2$ and $\ell$ is odd, and $\delta=0$ otherwise, apart from
the exceptions $(p, k, \ell) = (2, 5, 2)$, $(3, 7, 2)$, for which we have
$\mu(\Gamma_{2, 5, 2}) = \frac{1}{2}$ and $\mu(\Gamma_{3, 7, 2}) = \frac{2}{5}$, and
\[
\mu(\Gamma^+_{p, k, \ell}) = 1-\frac{(p-1)\ell+2}{k}
\]
with the exceptions $\mu(\Gamma^+_{2, 5, 2}) = \frac{1}{3}$ and $\mu(\Gamma^+_{3, 7, 2})
= \frac{1}{3}$. 
\end{Theo}
In a second step we have to show that the growth coefficient of any
virtually free group is rational. This is implied by the following
result.
\begin{Theo}
\label{thm:Rational}
Let $\Gamma$ be a finitely generated virtually free group, represented by a
finite graph of finite groups $\mathcal{G}$. Then the growth
coefficient of $\Gamma$ is the solution of a linear optimization problem
with rational coefficients which can effectively be computed from
$\mathcal{G}$. In particular, $\mu(\Gamma)$ is rational.
\end{Theo}

To prove Theorem~\ref{thm:Amalgam} we have to compute the growth
coefficient of a free product with a cyclic group of prime order
amalgamated, which can be done in greater generality.
\begin{Prop}
\label{prop:Amalgam}
Let $G$ be a finite group, $x\in G$ an element of prime order $p$, and
suppose that $|G|>2p$. Set
$\Gamma=G\ast_{\langle x\rangle}G$, where $x$ is embedded in both copies of $G$ in the same
way. Denote the conjugacy class of $x$ by $[x]$. Then $\mu(\Gamma)$ is given as
\begin{equation}
\label{eq:PropAm}
\mu(\Gamma) = \frac{1}{p} + \max_{H<G} \left(\frac{(|[x]\cap H|)(p-1)}{|[x]|p} - \frac{2}{(G:H)}\right).
\end{equation}
\end{Prop}

As another application of Proposition~\ref{prop:Amalgam} we can compare
the subgroup growth of an amalgam with its free subgroup growth. As
was shown by M\"uller\cite{Mvirtfrei}, the free subgroup growth is
determined by the Euler characteristics, more precisely, defining
$\mu_f(\Gamma)=\limsup_{n\to\infty}\frac{\log s_n^f(\Gamma)}{n\log n}$, where $s_n^f(\Gamma)$
is the number of free subgroups of index $n$ in $\Gamma$, we have for every
virtually free group $\Gamma$ the relation $\mu_f(\Gamma)=-\chi(\Gamma)$. Obviously,
we always have $\mu(\Gamma)\geq\mu_f(\Gamma)$. For the case that $N$
is normal in both $G$ and $H$ it is shown in \cite{normal} that
$\mu(G\ast_N H)=\mu((G/N)\ast(H/N))$, while 
$\chi(G\ast_N H)=\frac{1}{|G|}+\frac{1}{|H|}-\frac{1}{|N|}$ clearly
depends on $N$. This leads to the question under which 
conditions we have $\mu(\Gamma)=-\chi(\Gamma)$. A first answer is
given by the following.

\begin{Prop}
\label{prop:virtfrei}
Let $G$ be a finite group, $x\in G$ an element of prime order $p$, and
set $\Gamma=G\ast_{C_p}G$, where a generator of the cyclic group is mapped
onto $x$ in both copies of $G$. Then $\mu(\Gamma)=-\chi(\Gamma)$ holds true if and
only if $|N_G(\langle x\rangle)|\leq 2p$.
\end{Prop}

The algorithm to compute $\mu(\Gamma)$ from a graph of finite groups with
fundamental group $\Gamma$ is about as difficult to perform as compiling a
complete list of all subgroups of the vertex groups of the
graph, which becomes a difficult problem already for relatively small
groups as $S_8$, which has 151221 subgroups. However, if this
information is available, solving the resulting linear system is quite
easy. 

\begin{Prop}
\label{Prop:Example}
Let $\Gamma$ be the fundamental group of the graph of groups,
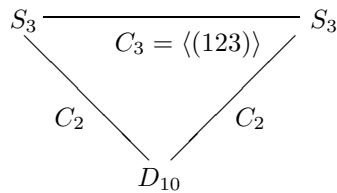
\begin{center}
\begin{picture}(8, 5)
\put(2, 4){$S_3$}
\put(6, 4){$S_3$}
\put(2.45, 4.15){\line(1, 0){3.4}}
\put(3.4, 3.7){$C_3=\langle(123)\rangle$}
\put(2.2, 3.9){\line(1, -1){1.65}}
\put(5.8, 3.9){\line(-1, -1){1.65}}
\put(3.7, 1.9){$D_{10}$}
\put(2.6, 2.7){$C_2$}
\put(5, 2.7){$C_2$}
\end{picture}
\end{center}
where the embeddings of $C_2$ identify one reflection in $D_{10}$ with
the element $(12)$ in one symmetric group, and another reflection with
$(13)$ in the other symmetric group. Then we have $\mu(\Gamma)=\frac{3}{2}$
and $\chi(\Gamma)=-\frac{9}{10}$.
\end{Prop}
Our second example is a simpler graph of groups with a more
complicated amalgamated subgroup.
\begin{Prop}
\label{Prop:Example2}
Let $\Gamma$ be the amalgamated product $S_4\ast_{C_2\times C_2} S_4$, where the
subgroups $\{1, (12)(34), (13)(24), (14)(23)\}$ and $\{1, (12), (34),
(12)(34)\}$ are identified. Then $\mu(\Gamma)=\frac{1}{4}$ and $\chi(\Gamma)=-\frac{1}{6}$.
\end{Prop}
\section{Some basic observations}
\begin{Lem}
\label{Lem:Transform}
Let $\Gamma$ be a finitely generated group. Then we have
\[
\mu(\Gamma) = \limsup_{n\to\infty} \frac{\log|\Hom(\Gamma, S_n)|}{n\log n} -1.
\]
\end{Lem}
\begin{proof}
Denote the right-hand side of the equation to be proven by $\mu_h(\Gamma)$,
and let $t_n(\Gamma)$ be the number of homomorphisms $\Gamma\to S_n$ with
transitive image. Then $s_n(\Gamma)=\frac{t_n(\Gamma)}{(n-1)!}$, which implies
$\mu(\Gamma)\leq\mu_h(\Gamma)$. For the reverse inequality suppose that $\mu(\Gamma)\leq\mu_h(\Gamma)-\delta$
for some $\delta>0$, and let $C$ be a constant such that $s_n(\Gamma)<C
n!^{\mu_h(\Gamma)-\delta/2}$ holds true for all $n\geq1$. Let $\Omega_1, \ldots, \Omega_k$ be a
partition of the set $\{1, \ldots, n\}$. Then the number of homomorphisms
$\varphi:\Gamma\to S_n$ with image having precisely the sets $\Omega_i$ as orbits equals
\[
\prod_{i=1}^k t_{|\Omega_i|}(\Gamma) = \prod_{i=1}^k (|\Omega_i|-1)! s_{|\Omega_i|}(\Gamma) \leq
C^k\prod_{i=1}^k (|\Omega_i|!)^{1+\mu_h-\delta/2}. 
\]
Let $\lambda=(\lambda_1, \ldots, \lambda_k)$ be a partition of $n$. Then the number of
homomorphisms $\varphi:\Gamma\to S_n$ such that the orbits of the image have sizes
$\lambda_1, \ldots, \lambda_k$ equals the quantity above multiplied by the number of
ways to partition a set of $n$ elements appropriately, this number in
turn is bounded above by $\binom{n}{\lambda_1, \ldots, \lambda_k}$, and we find that
the number of homomorphisms associated to a partition $\lambda$ is at most
\[
\binom{n}{\lambda_1, \ldots, \lambda_k} C^k\prod_{i=1}^k (\lambda_i!)^{1+\mu_h-\delta/2} = n! C^k
\prod_{i=1}^k (\lambda_i!)^{\mu_h-\delta/2} \leq n!^{1+\mu_h-\delta/2} C^k.
\]
Since the number of partitions grows subexponentially with $n$, we
deduce that the number of all homomorphisms $\varphi:\Gamma\to S_n$ is bounded
above by $n!^{1+\mu_h-\delta/2} e^{cn}$, which yields the contradiction $\mu_h\leq \mu_h-\delta/2$.
\end{proof}

We now show how to deduce Theorem~\ref{thm:Main} from
Theorem~\ref{thm:Amalgam} and \ref{thm:Rational}. Obviously, it
suffices to show that for each rational number $\frac{a}{b}$ with $a, b>0$
there exists a virtually free group $\Gamma$ with
$\mu(\Gamma)=\frac{a}{b}$. Suppose first that $0<\frac{a}{b}<1$, and that $b-a$
is odd. Let $p> b$ be a prime number satisfying $(p-1, b-a)=1$. Such a
prime number exists, since $b-a$ is odd, and therefore the residue
class $2\pmod{b-a}$ contains infinitely many primes, each of which
suffices for our purpose. Now
$p-1$ is invertible modulo $b-a$, let $1\leq\ell\leq b-a-1$ be the
integer satisfying $(p-1)\ell\equiv -1\pmod{b-a}$. Finally, set
$k=\frac{b\ell(p-1)+b}{b-a}$. Provided that $\ell p\leq k-2$, we can apply
Theorem~\ref{thm:Amalgam} to find 
\[
\mu(\Gamma_{p, k, \ell}) = 1 - \frac{(b-a)(p-1)\ell+b-a}{b\ell(p-1)+b} = \frac{a}{b}.
\]
To bound $\ell p$, note that
\[
\frac{(p-1)\ell+1}{k} = 1-\frac{a}{b}
\]
thus
\[
\frac{pl}{k}\leq \frac{p(b-a)}{(p-1)b} \leq 1-\frac{pa-b}{k} \leq 1-\frac{2}{k},
\]
that is, the condition $pl\leq k-2$ is indeed satisfied.

Next, suppose that $b-a$ is even, and that $b$ is odd. Let $p>b$ be a
prime number such that $(p-1, b-a)=2$, let $1\leq \ell\leq b-a-1$ be a solution of
the congruence $\ell(p-1)\equiv -2\pmod{b-a}$, and set
$k=\frac{b\ell(p-1)+2b}{b-a}$. Then we argue as above to find
\[
\mu(\Gamma_{p, k, \ell}) = 1 - \frac{(b-a)(p-1)\ell+2(b-a)}{b\ell(p-1)+2b} = \frac{a}{b}.
\]
Hence, our claim holds true for all rational numbers in $(0,
1)$. Now let $r$ be an integer, $\frac{a}{b}\in(0, 1)$ a rational
number, and $\Gamma_{a, b}$ a virtually free group with $\mu(\Gamma_{a,
  b})=\frac{a}{b}$. Consider the free product
$\Gamma=\Gamma_{a, b}\ast F_r$. We have
\[
|\Hom(\Gamma, S_n)| = |\Hom(\Gamma_{a, b}, S_n)|\cdot|\Hom(F_r, S_n)| =
n!^r|\Hom(\Gamma_{a, b}, S_n)|,
\]
thus $\mu(\Gamma)=r+\mu(\Gamma_{a, b})=r+\frac{a}{b}$, and we find virtually free groups with
growth coefficient $\frac{a}{b}$ for all non-integral positive
rational numbers. Finally, all integral growth coefficients are
realized by free groups, and our theorem is proven.

\section{Proof of Theorem~\ref{thm:Rational}} 
\begin{Lem}
\label{Lem:Stirling}
For integers $m\leq n$ we have $m!=n!^{n/m}e^{\mathcal{O}(n)}$.
\end{Lem}
\begin{proof}
By Stirling's formula we have
\[
m! = m^m e^{\mathcal{O}(m)} = (n^n)^{m/n} (m/n)^m e^{\mathcal{O}(m)} =
n!^{m/n} (m/n)^m e^{\mathcal{O}(m)}, 
\]
thus, it suffices to show that $m\log n/m =
\mathcal{O}(m)$. Differentiating with respect to $m$ shows that the
left hand side is maximal for $m=n/e$, and our claim follows.
\end{proof}

To every finitely generated virtually free group $\Gamma$ we can associate
a finite graph $\mathcal{G}$ of finite groups, that is, a connected directed
graph $(V, E)$, such that for each edge $e=(x, y)$ the reverse edge
$\bar{e}=(y, x)$ is in $E$, together with a collection of finite groups
$G_v$, $G_e$, $v\in V$, $e\in E$ satisfying $G_e=G_{\overline{e}}$ and
embeddings $\psi_e:G_e\to G_v$ for each 
edge $e$ with target $v$. We define the fundamental group of
$\mathcal{G}$ as follows: Let $T$ be a spanning tree of the underlying
graph with all reverse edges added. Then let $\Gamma$ be the group
generated by all $G_v, v\in V$ and symbols $g_e$ for all $e\in E$ subject
to the relations
\begin{enumerate}
\item $g_e=1$ for all $e\in T$,
\item $g_e=g_{\bar{e}}$ for all $e\in E$,
\item $g_e\psi_e(a)g_e^{-1} = \psi_{\bar{e}}(a)$ for all $e\in E$, $a\in G_e$.
\end{enumerate}
We remark without proof that the fundamental group does not depend on
the choice of the tree $T$, and that every finitely generated
virtually finite group occurs as fundamental group of some finite
graph of finite groups, for these and related results we refer the
reader to \cite{Serre}.

Note that each vertex group is embedded in $\Gamma$, thus, each
representation $\varphi:\Gamma\to S_n$ induces representations $\varphi_v: G_v\to S_n$. On
the other hand, not all tuples of representations $\varphi_v: G_v\to S_n$ can
be lifted to representations of the whole group. For each tuple of
representations $\{\varphi_v: G_v\to S_n, v\in V\}$, and every edge $e=(x,y)$,
the embeddings $\psi_e:G_e\to G_x$, $\psi_{\bar{e}}:G_e\to G_y$ induce two
representations of $G_e$, which obviously have to coincide to allow
for a lift. On the other hand, it is
obvious from the list of relations given above, that every tuple of
representations of the $G_v$ satisfying
these conditions lift to representation of $\Gamma$; in this case we will
call the tuple $(\varphi_v:G_v\to S_n, v\in V)$
admissible. Note that admissibility only depends on the equivalence
classes of the representations $\varphi_v$; we will refer to these data as
the representation type 
of the tuple. The number of
different equivalence classes of permutation representations of degree
$n$ of a
finite group equals the number of non-negative integral solution
$(x_1, \ldots, x_{m_v})$ of the equation $\sum_{i=1}^{m_v} x_i n_{v,
  i} = d$, that is, the number of lattice points in a certain domain,
which is contained in $[0, n]^{m_v}$. In particular, the number of
different equivalence classes grows at most polynomially, and
the value of $\mu(\Gamma)$ does not change if we replace the
number of all permutation representations by the maximal number of
permutation representations of $\Gamma$ with prescribed representation
type. Hence, to compute $\mu(\Gamma)$ we have to determine among all
admissible representation types the one leading to the greatest number
of representations, and to compute this number.

Permutation representations of $G_v$ are direct sums of transitive
representations, which in turn are equivalent to the coset action of
$G_v$. For each $v\in V$, we therefore consider a complete list of
transitive permutation representations $\rho_{v, i}, 1\leq i\leq m_v$ of
$G_v$, and let $n_{v,i}$ be the number of points $\rho_{v, i}$ acts upon.
Then the equivalence class of $\varphi_v:G_v\to S_n$ is determined by the
multiplicities $(\xi_{v, i}, 1\leq i\leq m_v)$ with which $\rho_{v, i}$
occurs in $\varphi_v$. A given representation type 
$(\xi_{v, i})_{\underset{1\leq i\leq m_v}{v\in V}}$ is admissible, if for each
edge $(x, y)\in E$, the actions of $G_e$ induced from the action of
$G_x$ and of $G_y$ coincide. To check this, note that each transitive
action of $G_x$ gives rise to a permutation representation of $G_e$,
which in turn decomposes into transitive $G_e$-actions. Hence, if
$\rho_{e, i}, 1\leq i\leq m_e$ is a complete list of transitive permutation
representations of $G_e$, the 
multiplicities of the transitive actions of $G_e$ in the action
induced from $G_x$ is the image of the vector $(\xi_{x, i})_{1\leq i\leq m_x}$
under a certain linear map given by a $m_e\times m_x$-matrix with
non-negative integral entries, and the condition of admissability
becomes a system of $m_e$ linear equations between the $m_x+m_y$
variables $\xi_{x, i}, \xi_{y, j}, 1\leq i\leq m_x, 1\leq j\leq m_y$.

Hence, we have shown that admissability of a representation is
equivalent to a certain set of linear equations with integral
coefficients in the variables $(\xi_{v, i})_{\underset{1\leq i\leq m_v}{v\in V}}$.

Next, we compute the number of representations for a given admissible
representation type. To do so, we first choose representations of the
finite groups $G_v$. For each $i\leq m_v$ we have to choose $\xi_{v, i}$
orbits of size $n_{v, i}$, the number of ways to do so is
\[
\frac{n!}{\prod_{i=1}^{m_v} (\xi_{v, i})! n_{v, i}!^{\xi_{v, i}}}.
\]
An action of $G_v$ similar to $\rho_{v,
  i}$ is equivalent to the action of $G_v$ on the cosets of a subgroup
$U_i$ of index $n_{v, i}$. Hence, to construct an action of $G_v$
similar to $\rho_{v, i}$ on a given set of $n_{v, i}$ points, we have
to identify these points with the cosets $G_v/U_{v, i}$. Two
identifications give the same representation if and only if one is
obtained from the other by shifting the cosets by some element $g\in
G_v$, that is, to count different actions we have to fix one point
which we identify with the coset $U_{v, i}$ itself. Hence, there are
$(n_{v, i}-1)!$ different actions equivalent to 
$\rho_{v, i}$, and we conclude that the number of representations of
$G_v$ of the given type equals 
\[
\frac{n!}{\prod_{i=1}^{m_v} (\xi_{v, i})! n_{v, i}^{\xi_{v, i}}}.
\]
We assume that the representation type is admissible, thus, for each
edge $e$ there exist $g_e$ such that relations of type (3)
hold. However, the relations of the first type require that the
edge groups do not act similarly, but are actually equal. Hence, the
number of representations of the fundamental group of the graph of
groups consisting only of the vertices in $T$ with given
representation type equals
\begin{equation}
\label{eq:elldef}
\prod_{v\in V}\frac{n!}{\prod_{i=1}^{m_v} (\xi_{v, i})!
n_{v,i}^{\xi_{v, i}}} \prod_{(x, y)\in T} 
\left(\frac{n!}{\prod_{i=1}^{m_e} \ell_{e, i}(\xi_{x, 1}, \ldots, \xi_{x, m_v})! 
n_{e, i}^{\ell_{e, i}(\xi_{x, 1}, \ldots, \xi_{x, m_v})}}\right)^{-1},
\end{equation}
where $\ell_{e, i}$ is the linear form giving the multiplicity of
$\rho_{e,i}$ in the restriction of $\sum_{i=1}^{m_v} \xi_{v, i}\rho_{v, i}$ to
$G_e$. Obviously, $\ell$ is a positive linear form with integral coefficients.

Finally, we have to choose the image of $g_e, e\in E\setminus T$. To
do so we have 
to respect relations of the third type, that is, the number of
possibilities to choose $g_e$ for $e$ equals the centralizer of a
representation of $G_e$ of the type induced by the type of $G_x$. This
number equals
\[
\prod_{i=1}^{m_e} \ell_{e, i}(\xi_{x, 1}, \ldots, \xi_{x, m_v})! 
n_{e, i}^{\ell_{e, i}(\xi_{x, 1}, \ldots, \xi_{x, m_v})},
\]
thus the total number of representations of $\Gamma$ with representation
type $(\xi_{v, i})_{\underset{1\leq i\leq m_v}{v\in V}}$ equals
\[
n!^{|E\setminus T|}\prod_{v\in V}\frac{n!}{\prod_{i=1}^{m_v} (\xi_{v,i})!
n_{v, i}^{\xi_{v, i}}} \prod_{(x, y)\in E}
\left(\frac{n!}{\prod_{i=1}^{m_e} \ell_{e, i}(\xi_{x, 1}, \ldots, \xi_{x, m_v})! 
n_{e, i}^{\ell_{e, i}(\xi_{x, 1}, \ldots, \xi_{x, m_v})}}\right)^{-1}.
\]
We can now apply Lemma~\ref{Lem:Stirling} together with the fact that
the number of edges in $T$ equals $|V|-1$ to see that the logarithm of
this quantity is equal to
\[
n\log n\left(1 - \sum_{v\in V}\sum_{i=1}^{m_v} \xi_{v, i}/n + \sum_{(x, y)\in E}
  \ell_{e, i}(\xi_{x, 1}, \ldots, \xi_{x, m_v})/n\right) + \mathcal{O}(n).
\]
We are looking for integers $\xi_{v, i}$, such that this expression
becomes maximal. To do so, we first determine real numbers $\alpha_{v, i}$
subject to conditions $\alpha_{v, i}\geq 0$, $\sum_{i=1}^{m_v}\alpha_{v, i}=1$, and
the equations implying admissability, such that the linear form
\[
- \sum_{v\in V}\sum_{i=1}^{m_v} \alpha_{v, i} + \sum_{(x, y)\in E} \ell_{e, i}(\alpha_{x, 1}, \ldots, \alpha_{x, m_v})
\]
becomes maximal. This optimization problem can be solved using only
matrix operations, hence, the solution is rational, since both the
linear constraints as well as the function to be maximized are
rational. Let $m$ be the least common multiple of all denominators
occurring in the optimum, and let $\mu_0$ be the value of the
problem. If $n$ is divisible by $m$, this rational solution gives rise
to an integral solution for the $\xi_{v,i}$, and therefore to the right
number of representations, which gives the lower bound $\mu(\Gamma)\geq\mu_0$. On
the other hand, for integral values of $\xi_{v, i}$, the value of the
linear problem is always bounded by the maximum over all real values,
and we deduce $\mu(\Gamma)=\mu_0$, which proves theorem~\ref{thm:Rational}.

\section{Proof of Propositions~\ref{prop:Amalgam} and ~\ref{prop:virtfrei}}

\begin{Lem}
Let $G$ be a finite group, $x\in G$  an element of order $p$, $U<G$ a subgroup,
$\varphi:G\to S_{(G:U)}$ the coset representation associated to $U$. Then the
multiplicity of the trivial representation in the restriction of $\varphi$
to $\langle x\rangle$ is $\frac{|G|\cdot|[x]\cap U|}{|U||[x]|}$.
\end{Lem}
\begin{proof}
A coset $aU$ is fixed by $x$ if and only if $x^a\in U$. For each element
in $U$ conjugate to $x$, there are $C_{G}(x)$ possible $a$
conjugating $x$ onto this element, and there are $|U|$ different $a$
defining the same coset, hence our claim.
\end{proof}

\begin{proof}[Proof of Proposition~\ref{prop:Amalgam}]
Since the amalgamated subgroup has only two inequivalent transitive
permutation representations, the condition of admissability reduces to
the condition that the number of fixed points of the reduced actions
coincide. Hence, denoting by $U_1, \ldots, U_m$ a complete list of
subgroups of $G$, a representation type $(\xi_{i, j})_{\underset{1\leq j\leq
    m}{1\leq i\leq 2}}$ is admissible if and only if
\begin{equation}
\label{eq:CondAd}
\sum_{j=1}^m \xi_{1, j} \frac{|G|\cdot|[x]\cap U_j|}{|U_j||[x]|} = \sum_{j=1}^m \xi_{2,
  j} \frac{|G|\cdot|[x]\cap U_j|}{|U_j||[x]|}.
\end{equation}
Since there is only one edge we remove the subscript $e$ from the
linear function $\ell_{e, i}$ used in (\ref{eq:elldef}), as well as
the subscript $x$ from the variables $\xi_{x, i}$. Then $\ell_1(\xi_1,
\ldots, \xi_m)$ equals the number of trivial representations in the
restriction of a representation of $G$ similar to
$\bigoplus_{i=1}^m\xi_i\cdot\rho_i$ to $C_p$. By the previous Lemma we
obtain
\[
\ell_1(\xi_1, \ldots, \xi_m) = \sum_{i=1}^m \frac{|G|\cdot|[x]\cap
  U_i|}{|U_i||[x]|} \xi_i.
\]
Since the degree of the representation does not change when
restricting it to a subgroup, and $C_p$ has only two transitive
representations, we obtain
\begin{eqnarray*}
\ell_2(\xi_1, \ldots, \xi_m) = \frac{1}{p}\Big(\sum_{i=1}^m n_i
\xi_i - \ell_1(\xi_1, \ldots, \xi_m)\Big),
\end{eqnarray*}
since $\sum_{i=1}^m n_i\xi_i = n$ we deduce
\[
\ell_1(\xi_1, \ldots, \xi_m) + \ell_2(\xi_1, \ldots, \xi_m) =
\frac{n}{p} + \frac{p-1}{p}\left(\sum_{i=1}^m \frac{|G|\cdot|[x]\cap
  U_i|}{|U_i||[x]|} \xi_i\right).
\]
The summand $\frac{n}{p}$ clearly does not influence the optimization.
The logarithm of the first factor in (\ref{eq:elldef}) becomes
\[
n\log n - \log n\sum_{i=1}^m \xi_i + \mathcal{O}(n) = 
\log n\sum_{i=1}^m (n_i-1)\xi_i + \mathcal{O}(n),
\]
hence, the linear form to be maximized is given as
\[
f(\alpha_{1,1}\ldots, \alpha_{2,m}) = \sum_{j=1}^m (\alpha_{1,j} + \alpha_{2,j})\big(1-\frac{1}{(G:U_j)}\big)
 + \frac{p-1}{p} \sum_{j=1}^m \alpha_{1, j} \frac{|G|\cdot|[x]\cap U_j|}{|U_j||[x]|},
\]
our aim is to show that the maximum is attained in a point satisfying
$\alpha_{1, j_0}=\alpha_{2, j_0}=1$, $\alpha_{i, j}=0$ for $j\neq j_0$
for some $j_0$. Note first that if in the second summand of the
definition of $f$ we replace all $\alpha_{1, j}$ by $\alpha_{2, j}$,
the value of $f$ does not change in view of (\ref{eq:CondAd}), and the
first summand is of the form $g(\alpha_{1, 1}, \ldots, \alpha_{1, m})
+ g(\alpha_{2, 1}, \ldots, \alpha_{2, m})$ for a linear form $g$,
thus, we may assume without loss that $\alpha_{1,j}=\alpha_{2,j}$ for
all $j$. Under this assumption (\ref{eq:CondAd}) is trivially
satisfied. Thus, we consider the maximum of a linear form on the set
$\{(\beta_1, \ldots, \beta_m): \beta_i\geq 0, \sum\beta_i = 1\}$,
which is either 0, or attained in a vertex $(0, \ldots, 0, 1, 0,
\ldots, 0)$. In the latter case our claim is proven, whereas in the
former we would have $\mu(\Gamma)=0$. However, as M\"uller
\cite{Mvirtfrei} showed, the number of free subgroups of $\Gamma$ of
index $n|G|$ is $n!^{-\chi(\Gamma)|G|+o(1)}$. In general, the Euler
characteristics of an amalgamated product $G\ast_U H$ is given by
$\frac{1}{|G|}+\frac{1}{|H|}-\frac{1}{|U|}$, and, using the assumption
$|G|\geq 3p$, we obtain the contradiction 
\[
0 = \mu(\Gamma) \geq \frac{- \chi(\Gamma)}{|G|} =\frac{1}{p|G|}-\frac{2}{|G|^2} \geq
\frac{1}{3p|G|}.
\]
\end{proof}
\begin{proof}[Proof of Proposition~\ref{prop:virtfrei}]
The Euler
characteristics of $G\ast_{C_p} G$ equals
$-\frac{1}{p}+\frac{2}{|G|}$. Set $H=C_p$ in
Proposition~\ref{prop:Amalgam} to obtain
\[
\mu(\Gamma) \geq \frac{1}{p} + \frac{|[x]\cap\langle x\rangle|(p-1)}{|[x]|p} - \frac{2p}{|G|}.
\]
Now $|[x]\cap\langle x\rangle|=(N_G(\langle x\rangle):C_G(x))$, and $|[x]|=(G:C_G(x))$, thus
$\mu(\Gamma)=-\chi(\Gamma)$ implies 
\[
\frac{(N_G(\langle x\rangle):C_G(x))(p-1)}{(G:C_G(x))p} \leq \frac{2p-2}{|G|},
\]
which is equivalent to $|N_G(\langle x\rangle)|\leq 2p$. To prove the
converse, suppose that $|N_G(\langle x\rangle)|\leq 2p$, and that the
subgroup $H$ of $G$ maximizing the right-hand side of
(\ref{eq:PropAm}) is not trivial. Since every subgroup not containing
an element conjugate to $x$ cannot maximize this expression, we have
in particular that $|H|\geq p$. Suppose first that $p\geq 3$. Then
$|N_G(\langle x\rangle)|<p^2$, thus a $p$-Sylow subgroup $P$ of $G$ is
cyclic of order $p$. For suppose otherwise. Then either $x$ is central
in $P$, which would imply $|C_G(\langle x\rangle)|\geq|P|\geq p^2$, or
$x$ is not central in $P$, in which case there is a central element
$y\in P\setminus\langle x\rangle$, which would imply $|C_G(\langle
x\rangle)|\geq|\langle x, y\rangle|\geq p^2$. Hence, in each $p$-Sylow
subgroup of $H$, there are $|[x]\cap\langle x\rangle|=(N_G(\langle
x\rangle):\langle x\rangle)$ elements conjugate to $x$. Similarly, if
$p=2$ and $|C_H(x)|$ is cyclic of order 2 or 4, there is one conjugate
of $x$ in every 2-Sylow subgroup of $H$. The number of $p$-Sylow
subgroups of $H$ is $|H|/p$ at most, thus
\[
|[x]\cap H| \leq \frac{|H|}{p}|[x]\cap \langle x\rangle| \leq \frac{|H|}{p}(N_G(\langle x\rangle):\langle x\rangle),
\]
and the right hand side of (\ref{eq:PropAm}) is bounded above by
\begin{eqnarray*}
\frac{1}{p} + \frac{|H|(N_G(\langle x\rangle):\langle x\rangle)(p-1)}{(G:C_G(\langle x\rangle)) p^2} -
\frac{2}{(G:H)}
 & \leq & \frac{1}{p} + \frac{2|H|(p-1)}{|G| p} -
\frac{2|H|}{|G|}\\
 & \leq & \frac{1}{p} + \frac{4(p-1)}{|G|} -
\frac{4p}{|G|}
\end{eqnarray*}
since $\frac{p-1}{p}<2$ and $|H|\geq2p$.

It remains to consider the case $p=2$, $C_H(x)\cong C_2\times C_2$. Since
$C_H(x)\leq N_G(\langle x\rangle)$, and $|N_G(x)|\leq 2p=4$, we deduce $N_G(\langle
x\rangle)=C_H(x)$, which implies $|C_G(x)|=4$. We therefore obtain $|[x]\cap
H|\leq |[x]| = \frac{|G|}{4}$, and with this bound the right hand side
of (\ref{eq:PropAm}) becomes
\[
\frac{1}{2} + \frac{|H|/4}{2|G|/4} - \frac{2}{(G:H)} = \frac{1}{2} -
\frac{3|H|}{2|G|} \leq \frac{1}{2} - \frac{6}{|G|} < -\chi(\Gamma),
\]
and the proposition is proven.
\end{proof}

\section{Proof of Theorem~\ref{thm:Amalgam}}

We now prove Theorem~\ref{thm:Amalgam}.
We have to consider 3 cases, depending on whether $x$ is an odd
permutation or not, and whether we consider $S_k\ast_{C_p}S_k$ or
$A_k\ast_{C_p}A_k$. Note that in the latter case $x$ is necessarily even,
which is the reason we do not have 4 cases. The computations are the
same for all these cases, we only give them here for the case
$S_k\ast_{C_p}S_k$ and $x$ even, that is, $p\neq 2$ or $p=2$ and $\ell$ even.

To apply Proposition~\ref{prop:Amalgam}, we have to consider the maximum
taken over all subgroups of $S_k$.
We claim that this maximum occurs for $H=A_{k-1}$, except for certain
small values of $k$ which lead to the exceptions in the theorem. In
this case we have $\frac{2}{S_k:A_{k-1}}=\frac{1}{k}$, and
$|[x]\cap H|$ equals the number of elements with $\ell$ $p$-cycles in
$S_{k-1}$, thus, we obtain
\[
\frac{(|[x]\cap H|)}{|[x]|} = \frac{\frac{(k-1)!}{(k-\ell p-1)!
\ell!p^\ell}}{\frac{k!}{(k-\ell p)!\ell!p^\ell}} = \frac{k-\ell p}{k},
\]
and therefore the value
\[
\frac{1}{p}+ \frac{(k-\ell p)(p-1)}{pk} - \frac{1}{k} =
1-\frac{\ell(p-1)+1}{k}. 
\]
We want to show that for every other subgroup we obtain a smaller
value, that is, we have to show that for each $H<S_k$ different from
$A_{k-1}$ we have
\[
\frac{(|[x]\cap H|)(p-1)}{|[x]|p} - \frac{2}{(S_k:H)} < \frac{(k-\ell
  p)(p-1)}{pk} - \frac{1}{k}.
\]
We will do so in four steps: First, we consider intransitive
subgroups, then transitive, but imprimitive subgroups. In both cases
we know that $H$ is contained in some direct product or some wreath
product, and these structures are simple enough to give upper bounds
for $|[x]\cap H|$. Next, we consider primitive subgroups. We know on
one hand, that a proper primitive subgroup cannot be too large, on the
other hand, we can show by direct construction that $x$ cannot have
too many fixed points, thus, $|[x]|$ has to be large. Comparing these
estimates we obtain a contradiction, provided that $n\geq 24$. Finally,
we check all primitive groups of order 24 using GAP\cite{GAP}. The
quantity $\frac{(|[x]\cap H|)}{|[x]|}$ has a probabilistic
interpretation as follows: Choose a permutation $\pi\in[x]$ at
random. Then the probability for the event $\pi\in H$ equals
$\frac{(|[x]\cap H|)}{|[x]|}$. It turns out that this interpretation
is useful for computing upper bounds.

An in transitive subgroup which does not contain
$A_{k-1}$ has an invariant set $A$ with $2\leq|A|\leq n/2$, thus we have
to estimate the probability that a permutation with $\ell$ cycles of
length $p$ leaves a given set $A$ of size $m$ invariant. Denote this
probability with $P(k, \ell, m)$. Distinguishing the possibilities that a
given $p$-cycle is totally inside or outside $A$, we obtain the
recursion relation 
\[
P(k, \ell, m) \leq \Big(\frac{m}{k}\Big)^p P(k-p, \ell-1, m-p) + 
\Big(1-\frac{m}{k}\Big)^p P(k-p, \ell-1, m),
\]
which is valid for $k\geq p$ and $\ell\geq 1$. We use this relation to prove
$P(k, \ell, m)<\frac{k-\ell p}{k}$ for all $m$ in the range $2\leq m\leq k/2$ by
induction on $k$. Inserting the induction hypothesis we have to prove
the estimate
\[
\frac{k-\ell p}{k} > \Big(\frac{m}{k}\Big)^p \frac{k-\ell p}{k-p} + 
\Big(1-\frac{m}{k}\Big)^p \frac{k-\ell p}{k-p},
\]
that is
\[
1-\frac{p}{k} > \Big(\frac{m}{k}\Big)^p + \Big(1-\frac{m}{k}\Big)^p.
\]
The expression on the right hand side is monotonically decreasing in
$m\leq k/2$, thus, its maximal value is at $m=2$, and we see that it
suffices to prove
\[
1-\frac{p}{k} > \Big(\frac{2}{k}\Big)^p + \Big(1-\frac{2}{k}\Big)^p.
\]
Suppose that $k\geq2p$. Then the terms in the binomial expansion of
$(1-2/k)^p$ are alternating and of decreasing modulus, thus we obtain
an upper bound by deleting all terms except the first three, and find
that we only have to prove
\[
1-\frac{p+1}{k} > \Big(\frac{2}{k}\Big)^p + 1-\frac{2p}{k} + \frac{2p(p-1)}{k^2}.
\]
However, bounding the first summand on the right hand side by
$\frac{4}{k^2}$, and using the assumption $k>2p$, this inequality is
seen to hold true with the possible exceptions $p=2, k=4$, which can
be checked directly. Hence, the induction 
step is correct for $k\geq2p$. For $k<2p$ we have 
$\ell=1$, in which case we compute $P(k, \ell, m)$ directly to obtain
\begin{eqnarray*}
P(k, 1, m) & = & \frac{(k-m)(k-m-1)\cdots(k-m-p+1)}{k(k-1)\cdots(k-p+1)}\\
 & \leq & \frac{(k-2)(k-3)\cdots(k-p-1)}{k(k-1)\cdots(k-p+1)}\\
 & = & \frac{(k-p)(k-p-1)}{k(k-1)}\\
 & < & \frac{k-p}{k}.
\end{eqnarray*}
Hence, for intransitive subgroups our claim is proven.

Next, we consider imprimitive subgroups. Without loss we may assume
that $H=S_a\wr S_b$, and denote by $\pi:H\to S_b$ the canonical
projection.

We begin by considering some special cases, the first being the case
$p=b$. If $x\in H$, then either some $S_k$-conjugate of $x$
contains $a$ cycles of length $p$, which contradicts the assumption
$\ell p<k$, or there are $p$ domains of imprimitivity, each of which is an
invariant set under each element of $[x]\cap H$, and we can argue as in
the case of intransitive groups. 

Next, suppose that $a=p=2$, and denote by $P(b, \ell)$ the probability
that a random element $x\in S_k$ consisting of $\ell$ cycles of length $2$
is contained in $S_2\wr S_b$. Choose a domain of imprimitivity. Then this
domain is either fixed pointwise, or a 2-cycle of $x$, or there are
two 2-cycles of $x$ mapping this domain onto another domain of
imprimitivity. These cases are mutually exclusive, thus we can compute
the probability of each single case and obtain for $\ell\geq 2$ the
recursion relation
\begin{eqnarray*}
P(b, \ell) & = & \frac{(2b-2\ell)(2b-2\ell-1)}{2b(2b-1)} P(b-1, \ell)\\
&&+ \frac{1}{2b-1}\Big(1-\frac{(2b-2\ell)(2b-2\ell-1)}{2b(2b-1)}\Big)
P(b-1, \ell-1)\\ 
&&+ \frac{1}{2b-1}\Big(1-\frac{(2b-2\ell)(2b-2\ell-1)}{2b(2b-1)}\Big)\\
&&\qquad\qquad\times
\Big(1-\frac{(2b-2\ell-2)(2b-2\ell-3)}{(2b-2)(2b-3)}\Big) P(b-2, \ell-2)
\end{eqnarray*}
where $P(b-1, \ell)$ has to be interpreted as 0 if $b=\ell$.
Estimating all probabilities on the right hand side as well as the
brackets $(1-\ldots)$ by 1, we obtain the estimate
$P(b,\ell)<\frac{(2b-2\ell)(2b-2\ell-1)+4b}{2b(2b-1)}$, which implies $P(b,\ell)<
\frac{2b-2\ell-1}{2b}$, and hence our claim, provided that
$4b\ell+1>6b+4\ell^2$, which is the case for all $\ell$ unless one of
$b\leq7$, $\ell\leq 1$ or $\ell\geq b-1$ holds true. For $\ell=1$ we immediatelly see that
$P(b, 1)=\frac{1}{2b-1}$, which is small enough for all $b\geq 3$. For
$b=\ell$ we have
\[
P(b, b) = \frac{1}{2b-1}\big(P(b-1, b-1) + P(b-2, b-2)\big), 
\]
which implies by a simple induction that $P(b, b)<\frac{1}{b^2}$ for
all $b\geq 6$. For $P(b, b-1)$ the recursion implies
\[
P(b, b-1) < \frac{1}{b(2b-1)}P(b-1, b-1) + \frac{1}{2b-1}\big(P(b-1,
b-2) + P(b-2, b-3)\big).
\]
Explicite computation implies that $P(6, 5)<0.014$ and $P(7,
6)<0.0051$, together with the bound for $P(b, b)$ we obtain by
induction that $P(b, b-1)<\frac{1}{b^2}$ for all $b\geq 6$, which is
again acceptable. Thus, our claim is proven for all $b\geq 8$, and for
the remaining values of $b$ it can be checked directly using the
recurrence relation above.

Now we embark on the general case.
Let $x\in S_k$ be an element with $\ell$ cycles of length $p$
chosen at random; we estimate the probability that $x\in H$. We
distinguish two cases, depending on whether there exists a domain of
imprimitivity which is left invariant by $x$ or not. It suffices to
show that in each case the conditional probability for the event $x\in H$
is at most $\frac{k-\ell p-1}{k}$. Suppose that $x$ moves some domain of
imprimitivity. Then $x$ has $a$
cycles of length $p$, all of which 
are mapped onto the same element under $\pi$; leaving away all other
restrictions on $x$ we obtain an upper bound for the probability in
question. We get an even coarser bound, if we add the condition that
all points in the first domain of imprimitivity are actually moved,
and only ask whether the action on these other points is actually
mapped onto the same cycle under $\pi$. We therefore obtain the upper bound
\[
\frac{(a-1)!^{p-1}}{(k-p-a+1)(k-p-a)\cdots(k-pa+1)} 
= \frac{(a-1)!^{p-1}(k-pa)!}{(k-p-a+1)!}.
\]
It suffices to show that the probability in question is
$\leq\frac{1}{k}$. Compare the $p-1$ largest factors in the numerator,
that is, $a^{p-1}$ with the $p-1$ largest factors in the denominator, that
is, $(k-p-a+1)\cdots(k-a-2p+3)$, and note that $k=ab$. Comparing these
terms we obtain a contribution $< (b-1-2p/a)^{1-p}$, continuing in
this way we find that the probability in question is at most
\[
\big((b-1-2p/a)(b-2-2p/a)\cdots(b-a-2p/a)\big)^{1-p}.
\]
All factors are at most 1, thus, to give an upper bound, we may single
out any factor easily dealt with. If $b\geq2a+4p/a$, there
is one factor $\leq a^{-1}$, and two factors $\leq2/b$, hence, the whole
product is less than $4/(ab^2)\leq 1/(ab)$. 
Next, note that there are $p-1$ factors which are less than
$\frac{1}{2}$, thus, if $2^{p-1}>ab$, we are done. In view of the
previous result we see that we may assume $p\leq a$, provided that
$a\geq9$. Again, these special cases can be excluded by looking at the
original expression; we therefore obtain the restrictions $b<2a+4, p<a$.

We now go back to the original fraction and note that in the
denominator, there are $(p-1)(a-1)$ consecutive positive integers,
hence, we obtain the upper bound
\[
\frac{(a-1)!^{p-1}}{(ap-a-p+1)!}.
\]
This expression is decreasing as $p$ increases, for $p=3$ it becomes
\[
\frac{(a-1)!^2}{(2a-2)!} = \binom{2a-2}{a-1}^{-1} < \frac{1}{2a^2+4a}
< \frac{1}{ab},
\]
provided that $a\geq 6$, the remaining cases have to be checked
individually again. Hence, we may assume $p=2$, and the probability in
question is equal to
\[
\frac{(a-1)!}{(ab-a-1)\cdots(ab-2a+1)} \leq (b-1)^{1-a},
\]
which is $\leq \frac{1}{ab}$ unless $b=2$, $b=3$ and $a\leq6$, $b=4$ and $a\leq
4$, or $a=2$. However, the special cases dealt with at the beginning
imply that $a, b\geq 3$, and we are again lead to a few special cases
which are easily excluded separately.

Hence, for imprimitive $H$ we always have $\frac{|[x]\cap H}{|[x]|} \leq
\frac{k-2\ell-1}{k}$, and we do not have to consider such groups any more.

To deal with primitive subgroups, we use the following.

\begin{Lem}
Let $H<S_k$ be a primitive subgroup, and suppose that one of the
following two conditions hold true.
\begin{enumerate}
\item $H$ contains a $q$-cycle for some prime number $q<k-4$;
\item $|H|>3^k$ or $|H|>2^k$ and $k\geq 25$.
\end{enumerate}
Then $H$ is $A_n$ or $S_n$.
\end{Lem}
\begin{proof}
The first condition is due to (confer \cite[II. Satz 4.5]{Huppert}),
whereas the second was proven by Mar\'oti\cite{Maroti}.
\end{proof}

Suppose that $H<S_n$ is primitive, not equal to $S_n$ or $A_n$, and
that $\frac{|[x]\cap H|}{|[x]|}\geq\frac{k-\ell p-1}{k}$. Define an equivalence
relation on $[x]$ as follows: Two elements $y, z$ are equivalent, if the $\ell$
sets of $p$ points in the $p$-cycles of $y$ and $z$ are equal, and if
the action of $y$ and $z$ is the same on each of these sets with the
exception of the one containing the least number moved by $y$ and
$z$, and on this set, there is some $i$ such that the action of $y^i$
is equal to the action of $z$; that is, $y$ and $z$ share $\ell-1$
$p$-cycles, and move the same points, and on the $p$ points on which
they differ, one is a power of the other. Suppose that
$\frac{k-\ell p-1}{k}>\frac{1}{p-1}$. Since every equivalence class
contains $p-1$ elements of $[x]$, there is one class containing at
least 2 elements of $[x]\cap H$. However, their quotient is a single
$p$-cycle, thus $H$ contains a $p$-cycle, which is only possible for
$p\geq k-4$, contradicting our assumption for $k\geq 7$. Hence, we either
have $k\leq 6$ or $\frac{k-\ell p-1}{k}\leq\frac{1}{p-1}$. For $p=2$ this
argument does not apply, while for $p=3, 5$ it is not very
efficient. For these cases we argue as follows. In $S_5$, the only
even permutations which do not consist of a single cycle of prime
length are the identity and the permutations conjugate to
$(12)(34)$. We call a set $A$ of elements in $S_5$ admissible, if the
subgroup generated by all elements $\sigma\tau^{-1}, \sigma, \tau\in
A$ are of this form, that is, if there is some conjugate of the Klein
four group containing all elements $\sigma\tau^{-1}$. By direct
inspection we find that an admissible set consisting of elements
conjugate to $(12)(34)$ contains at most 3 elements, an admissible set
of 3-cycles contains at most 4 elements, and an admissible set of
5-cycles contains at most 2 elements. We now form subsets of $[x]$ in
the following way: For $p=2$ we choose 5 points 
$\{a_1, \ldots, a_5\}$, and $\ell-2$ pairs of points $(x_1, y_1),
\ldots, (x_{\ell-2}, y_{\ell-2})$, such that all points $a_i, x_i,
y_i$ are distinct, and consider the set of all permutations which
interchange $x_i$ and $y_i$, and act as the product 
of two transpositions on $\{a_1, \ldots, a_5\}$. Similarly, for $p=3,
5$ we choose five points and $\ell-1$ cycles of length $p$, and form
the set of all permutations which act as prescribed on the $\ell-1$
cycles, and as an arbitrary cycle of length $p$ on $\{a_1, \ldots,
a_5\}$. We do so for all choices of the five points and the
cycles. The number $N$ of elements in one subset equals the number of
permutations conjugate to $(12)(34)$, $(123)$ or $(12345)$,
respectively, in $S_5$, that is, it equals 15, 20, or 24. On average,
the proportion of elements contained in $H$ among all elements in a
subset is the same as the proportion of elements in $H$ among all of
$[x]$, thus, there exists some subset containing at least
$\frac{|[x]\cap H|}{|[x]|}\cdot N$ elements in $H$. The quotient of
two elements in one subset moves at most 5 points, and if this
quotient acts as a cycle of prime length, we obtain $k\leq 9$ or that
$H$ is one of $A_n$, $S_n$. In other words, if $H\neq A_n, S_n$, the
restriction of the action of a subset to the five-point set is
admissible, and using the bounds for the size of admissible sets we
obtain $\frac{|[x]\cap H|}{|[x]|}\leq \frac{1}{5}$ for $p=2, 3$, and
$\frac{|[x]\cap H|}{|[x]|}\leq \frac{1}{12}$ for $p=5$, and therefore
the same bounds for $\frac{k-\ell p-1}{k}$.

Now suppose that $k> 24$. Then we have $|H|\leq 2^k$, and $|[x]\cap
H|\geq\frac{|[x]|}{k}$. As a function of $\ell$, $|[x]|$ is unimodal, thus
we have
\[
|[x]| \geq \min\Big(\frac{k!}{(k-p\ell_0)!\ell_0!
  p^{\ell_0}},\frac{k!}{(k-p\lfloor k/p\rfloor)!\lfloor k/p\rfloor! p^{\lfloor k/p\rfloor}} \Big), 
\]
where
\[
\ell_0=\begin{cases} \left\lceil\frac{(p-2)k-p+1}{p(p-1)}\right\rceil,
  & p\geq 7,\\
\left\lceil\frac{11k}{12}\right\rceil-1, & p=5\\
 \left\lceil\frac{4k}{5}\right\rceil-1, & p=2, 3
\end{cases}
\]
is the lower bound for $\ell$ obtained above. We therefore obtain the
inequality
\[
2^k\geq|H|\geq|[x]\cap
H|\geq\frac{|[x]|}{k} \geq \frac{1}{k}\min\Big(\frac{k!}{(k-p\ell_0)!\ell_0!
  p^{\ell_0}},\frac{k!}{(k-p\lfloor k/p\rfloor)!\lfloor k/p\rfloor! p^{\lfloor k/p\rfloor}} \Big),
\]
and a simple computation shows that this inequality does not hold true
for any $k\geq 25$.

Finally, for $k\leq 50$ the primitive groups were tabulated by
Sims\cite{Sims} and are readily available under a computer algebra
system such as GAP\cite{GAP}. There are 154 primitive groups of degree
$\leq 24$ which have to be checked individually. This work is feasible,
since among these groups there are 45 symmetric or alternating groups,
and most other groups can be discarded using the trivial estimate
$|[\pi]\cap G|\leq G$. There remain only the cases $C_2\ltimes C_5$ acting
on 5 points, and $PSL_3(\F_2)$ acting on 7 points, which lead to the
exceptions mentioned in the theorem.

\section{Computational examples}
In this section we prove Propositions~\ref{Prop:Example} and
\ref{Prop:Example2}, beginning with the former.

The following table lists the subgroups of $S_3$ together with the
number of fixed points of $(12)$ and $(123)$ under the coset action.

\begin{tabular}{c|c|c|c}
Generators & Order & Fixpoints of $(12)$ & Fixpoints of $(123)$\\ \hline
$(12), (13)$ & 6 & 1 & 1\\
$(123)$ & 3 & 0 & 2\\
$(12)$ & 2 & 1 & 0\\
$(13)$ & 2 & 1 & 0\\
$(23)$ & 2 & 1 & 0\\
id & 1 & 0 & 0
\end{tabular}

Similarly, for $D_{10}$, denoting by $\tau_1, \ldots, \tau_5$ the reflections, we
have the following. 

\begin{tabular}{c|c|c|c}
Generators & Order & Fixpoints of $\tau_1$ & Fixpoints of $\tau_2$\\ \hline
$\tau_1, \tau_2$ & 10 & 1 & 1\\
$\tau_1\tau_2$ & 5 & 0 & 0\\
$\tau_1$ & 2 & 1 & 1\\
$\tau_2$ & 2 & 1 & 1\\
$\tau_3$ & 2 & 1 & 1\\
$\tau_4$ & 2 & 1 & 1\\
$\tau_5$ & 2 & 1 & 1\\
id & 1 & 0 & 0
\end{tabular}

Denote by $x_1$, $x_2$, $x_4$ the number of orbits, on which one
symmetric group acts similar to the action on the cosets of $S_3$,
$A_3$, and $1$, respectively, and let $x_3$ be the number of orbits on
which the action is similar to one of the coset actions of one of the
2-element subgroups. Similarly, denote by $y_1, \ldots, y_4$ the number of
orbits for the other symmetric group. Finally, denote by $z_1, z_2,
z_4$ the number of orbits on which the action of $D_{10}$ is similar
to the action on the cosets of $D_{10}$, $C_5$ and $1$, respectively,
and by $z_3$ the number of orbits of size 2.

Then we obtain the following conditions:
\begin{eqnarray*}
x_1, x_2, x_3, x_4, y_1, y_2, y_3, y_4, z_1, z_2, z_3, z_4 & \geq & 0\\
x_1 + 2 x_2 + 3x_3 + 6x_4 & = & 1\\
y_1 + 2 y_2 + 3y_3 + 6y_4 & = & 1\\
z_1 + 2 z_2 + 5z_3 + 10z_4 & = & 1\\
x_1+2x_2 & = & y_1+2y_2\\
x_1+x_3 & = & z_1+z_3\\
y_1+y_3 & = & z_1+z_3
\end{eqnarray*}
The linear form to be maximized is
\begin{multline*}
\ell(x_1, x_2, x_3, x_4, y_1, y_2, y_3, y_4, z_1, z_2, z_3, z_4) =\\
-(x_1 + x_2 + x_3 + x_4 + y_1 + y_2 + y_3 + y_4 + z_1 + z_2 + z_3 + z_4)\\
+ (x_1 + 2x_2 + 2x_3 + 3x_4) + (y_1 + 2y_2 + 2y_3 + 3y_4) + (z_1 +
2z_2 + 3z_3 + 5z_4)
\end{multline*}
Since the domain described by the equations and inequalities is
convex, and this domain as well a sthe linear form $ell$ is symmetric
with respect to the reflection $x_i\leftrightarrow y_i$, we see that the maximum of
$\ell$ on this domain is the same as the maximum under the additional
constraint $x_i=y_i$. Therefore, we obtain the new system
\begin{eqnarray*}
x_1, x_2, x_3, x_4, z_1, z_2, z_3, z_4 & \geq & 0\\
x_1 + 2 x_2 + 3x_3 + 6x_4 & = & 1\\
z_1 + 2 z_2 + 5z_3 + 10z_4 & = & 1\\
x_1+x_3 & = & z_1+z_3\\
\end{eqnarray*}
with linear form 
\[
\ell'(x_1, x_2, x_3, x_4, z_1, z_2, z_3, z_4) =
2x_2 + 2x_3 + 4x_4 + z_2 + 2z_3 + 4z_4
\]
Increasing $x_2$ by $3\epsilon$, and decreasing $x_4$ by $\epsilon$ does not change
the conditions, but increases the value of $\ell'$ by $2\epsilon$, hence,
in the optimum we have $x_4=0$, and similarly we deduce $z_4=0$. If
$x_1$ and $x_2$ are increased by $\epsilon$, and $x_3$ is decreased by $\epsilon$,
the equations remain valid, and the linear form is not changed, thus,
we may assume that $x_3=0$. Similarly, if we increase $z_1$ by $\epsilon$,
$z_2$ by $2\epsilon$ and decrease $z_3$ by $\epsilon$, the value of $\ell'$ does
not change, and we obtain $z_3=0$, thus, $x_1=z_1$. We therefore
obtain the system
\begin{eqnarray*}
x_1+2x_2 & = & 1\\
x_1+2z_2 & = & 1\\
\end{eqnarray*}
with linear form
\[
\ell''(x_1, x_2, z_2) = 2x_2+z_2,
\]
which has obviously the optimal value $\mu= \frac{3}{2}$.

We now compute $\mu(S_4\ast_{C_2\times C_2}S_4)$ with amalgamation as described
in Proposition~\ref{Prop:Example2}.

We first give a complete list of subgroups of $S_4$ together with the
induced actions of the two subgroups. Let $\tau_1, \ldots, \tau_5$ be the
transitive actions of $C_2\times C_2$, where $\tau_1$ is trivial, $\tau_5$ is
regular, $\tau_2$ is the action on the cosets of $\{1, (12)(34)\}$, $\tau_3$
corresponds to $\{1, (13)(24)\}$ respectively $\{1, (12)\}$, and $\tau_4$
corresponds to $\{1, (14)(23)\}$ respectively $\{1, (34)\}$. We can
simplify our computations, since we are only dealing with one vertex
group, it is easy to see that in this case the induced action only
depends on the conjugacy class of the subgroup, not on the subgroup
itself. Thus, we only have to consider 11 distinct representations of
$S_4$, the relevant information is contained in the following table.

\begin{tabular}{l|c|c|c|c|c|c|c|c|c|c|c|c}
 & Subgroup & Index & \multicolumn{5}{c|} {Representation of $U_1$} &
 \multicolumn{5}{c} {Representation of $U_2$}\\ \hline
 & (Name or generators)& &$\tau_1$ & $\tau_2$ & $\tau_3$ & $\tau_4$ & $\tau_5$ &
 $\tau_1$ & $\tau_2$ & $\tau_3$ & $\tau_4$ & $\tau_5$ \\ \hline 
$\rho_1$ & $S_4$ & 1 & 1 & 0 & 0 & 0 & 0 & 1 & 0 & 0 & 0 & 0\\
$\rho_2$ & $A_4$ & 2 & 2 & 0 & 0 & 0 & 0 & 0 & 1 & 0 & 0 & 0\\
$\rho_3$ & (12), (13)(24) & 3 & 3 & 0 & 0 & 0 & 0 & 1 & 1 & 0 & 0 & 0\\
$\rho_4$ & (12)(34), (13)(24) & 6 & 6 & 0 & 0 & 0 & 0 & 0 & 3 & 0 & 0 &
0\\
$\rho_5$ & (12), (123) & 4 & 0 & 0 & 0 & 0 & 1 & 0 & 0 & 1 & 1 & 0\\
$\rho_6$ & (1234) & 6 & 0 & 1 & 1 & 1 & 0 & 0 & 1 & 0 & 0 & 1\\
$\rho_7$ & (12), (34) & 6 & 0 & 1 & 1 & 1 & 0 & 2 & 0 & 0 & 0 & 1\\
$\rho_8$ & (123) & 8 & 0 & 0 & 0 & 0 & 2 & 0 & 0 & 0 & 0 & 2\\
$\rho_9$ & (12)(34) & 12 & 0 & 2 & 2 & 2 & 0 & 0 & 2 & 0 & 0 & 2\\
$\rho_{10}$ & (12) & 12 & 0 & 0 & 0 & 0 & 6 & 0 & 0 & 2 & 2 & 1\\
$\rho_{11}$ & $\mathrm{id}$ & 24 & 0 & 0 & 0 & 0 & 6 & 0 & 0 & 0 & 0 & 6
\end{tabular}

From these table we find that our linear problem has 22 variables
$x_1, \ldots, x_{11}, y_1\ldots, y_{11}$, corresponding to the 11 representation
of each factor, and the equations
\begin{eqnarray*}
\sum_{i=1}^{11} n_i x_i & = & 1\\
\sum_{i=1}^{11} n_i y_i & = & 1\\
x_1+2x_2+3x_3+6x_4 & = & y_1+y_3+2y_7\\
x_6+x_7+2x_9 & = & y_2+y_3+3y_4+y_6+2y_9\\
x_6+x_7+2x_9 & = & y_5+2y_{10}\\
x_5+2x_8+6x_{10}+6x_{11} & = & y_6+y_7+2y_8+2y_9+y_{10}+6y_{11}
\end{eqnarray*}
where $n_i$ is the number of points $\rho_i$ acts upon. The linear form
to be optimized is
\begin{multline*}
\ell(x_1, \ldots, x_{11}, y_1, \ldots, y_{11}) = - \sum_{i=1}^{11} x_i - \sum_{i=1}^{11}y_i + \\
x_1+2x_2+3x_3+6x_4 + x_5+2x_6+2x_7+2x_8+4x_9+6x_{10}+6x_{11}.
\end{multline*}
Solving this linear program we find that the optimum is attained in
the point $x_4 = \frac{1}{18}, x_{11}=\frac{1}{36}$, $y_7=\frac{1}{6}$,
and all other variables equal to 0. The value of the form in this
point is $\frac{1}{4}$, thus, we find $\mu(\Gamma)=\frac{1}{4}$.

The computation of the Euler characteristics is straightforward in
both cases.

Note that in the second example the optimising tuple contains two
non-zero values $x_i$, that is, different from the case of free
products of finite groups it is not true that for almost all homomorphisms
$\varphi:\Gamma\rightarrow S_n$ the restriction to a finite subgroup $U$
decomposes into orbits on almost all of which $U$ acts in the same
way. This observation makes the existence of a simple relation between
$\mu(\Gamma)$ and group theoretic data of $\Gamma$ rather unlikely.

\end{document}